\documentclass[11pt, reqno]{amsart}

\usepackage[left=1in,right=1in,top=1in,bottom=1in]{geometry}
\usepackage{amsmath,amssymb,amsthm,mathtools,mathrsfs}
\usepackage[colorlinks=true,linkcolor=blue,citecolor=blue,urlcolor=blue]{hyperref}
\numberwithin{equation}{section}
\newtheorem{theorem}{Theorem}[section]
\newtheorem{proposition}[theorem]{Proposition}
\newtheorem{lemma}[theorem]{Lemma}
\newtheorem{corollary}[theorem]{Corollary}

\theoremstyle{remark}
\newtheorem{rem}[theorem]{Remark}

\newcommand{\R}{\mathbb R}
\newcommand{\Szero}{\mathbb S_0}
\newcommand{\tr}{\operatorname{tr}}
\newcommand{\diag}{\operatorname{diag}}
\newcommand{\ess}{\mathrm{ess}}

\title[Sharp instability of planar liquid-crystal interfaces]
{Sharp instability of planar isotropic--nematic interfaces\\
in the Landau--de Gennes model}

\author{Wei Wang}
\address{School of Mathematical Sciences, Peking University, Beijing 100871, China}
\email{2201110024@stu.pku.edu.cn,\,\,wwmath166@outlook.com}

\author{Qin Wu}
\address{School of Mathematical Sciences, Peking University, Beijing 100871, China}
\email{m15985481267@163.com}

\subjclass[2020]{}
\date{}

\begin{document}

\begin{abstract}
We study the stability of one-dimensional planar isotropic--nematic
interfaces in the Landau--de Gennes model with anisotropic elastic constant
$L$.  Earlier work proved the instability for $L<0$ only under an extra
condition.  We remove this condition and prove
that any non-negative minimizer of the reduced energy within the diagonal
planar class is unstable under general one-dimensional perturbations
throughout $-\frac{3}{2}<L<0$. This result
shows that $L=0$ is the sharp endpoint of the instability range from the
negative-$L$ side: the critical operator is non-negative at $L=0$, while
instability holds over the full negative-$L$ range of the reduced problem.
We also define the optimal stability index and correct a factor-of-two
normalization inconsistency in a previously stated interface profile.
\end{abstract}

\maketitle

\section{Introduction}

An isotropic--nematic interface separates a disordered phase from a phase
with orientational order.  Both the strength and the direction of this order
may vary across the transition layer, and states with three distinct
eigenvalues, called biaxial states, may occur there.  A
director field alone cannot represent all of these effects.  We therefore
use the Landau--de Gennes order parameter $\mathbf Q$, a symmetric traceless tensor
\cite{StephenStraley1974,Wang2021}.  This model respects the head--tail
symmetry of nematic molecules and permits changes in the degree of order.
We write the state space as
\[
\Szero:=\{\mathbf Q\in\R^{3\times3}\colon
\mathbf Q=\mathbf Q^{\mathrm{T}},\ \tr\mathbf Q=0\}.
\]
Other descriptions are useful at different scales. The Onsager theory uses an
orientational probability density \cite{Onsager1949}. The Oseen--Frank model
uses a director field \cite{Frank1958,Oseen1933}, and Ericksen's theory adds a
scalar order parameter \cite{Ericksen1991}.  The tensor model is well suited
to the interface stability problem considered here.

The Landau--de Gennes energy contains a bulk term and an elastic term.  The
bulk term selects the preferred phases, while the elastic term penalizes
spatial variation.  A standard form is
\begin{align*}
\mathcal F(\mathbf Q,\nabla \mathbf Q):=&\int_{\Omega}\underbrace{\frac{1}{2}\left(L_1\partial_k\mathbf Q_{ij}\partial_k\mathbf Q_{ij}+L_2\partial_j\mathbf Q_{ij}\partial_k\mathbf Q_{ik}+L_3\partial_k\mathbf Q_{ij}\partial_j\mathbf Q_{ik}+L_4\mathbf Q_{ij}\partial_i\mathbf Q_{kl}\partial_j\mathbf Q_{kl}\right)}_{f_E(\mathbf Q,\nabla\mathbf Q):\text{elastic energy}}\mathrm{d} x\\
&+\int_{\Omega}\underbrace{\frac{a}{2}\tr(\mathbf Q^2)-\frac{b}{3}\tr(\mathbf Q^3)+\frac{c}{4}(\tr(\mathbf Q^2))^2}_{f_B^{a,b,c}(\mathbf Q):\text{bulk energy}}\mathrm{d} x,
\end{align*}
Here, $\Omega\subset\R^3$ is open and $\mathbf Q=(\mathbf Q_{ij})$ is a symmetric traceless
$3\times 3$ matrix. The coefficients $a,b,c$ specify the bulk potential and
depend on the material and the temperature.  The constants $L_i$ describe the
elastic response.  In the small-elastic-length regime, the scaled energy \cite{MZ10} is
\[
\mathcal F^\varepsilon(\mathbf Q,\nabla \mathbf Q)=\int_\Omega f_E(\mathbf Q,\nabla\mathbf Q)\mathrm{d} x+\frac{1}{\varepsilon^2}\int_\Omega f_B^{a,b,c}(\mathbf Q)\mathrm{d} x,
\]
where $0<\varepsilon\ll1$. For this regime, see also related works \cite{Can15,Can17,FWW25,FWW26,GZ23,Nguyen2010}. The $L_4$-term may make the
energy unbounded from below \cite{Ball2010}.  We therefore set $L_4=0$.
Integration by parts also gives
\[
\begin{aligned}
&\int_\Omega\left(L_2\partial_j\mathbf Q_{ij}\partial_k\mathbf Q_{ik}+L_3\partial_k\mathbf Q_{ij}\partial_j\mathbf Q_{ik}\right)\mathrm{d} x\\
&\quad=\int_\Omega(L_2+L_3)\partial_j\mathbf Q_{ij}\partial_k\mathbf Q_{ik}\mathrm{d} x+\text{boundary terms}.
\end{aligned}
\]
Only the sum $L_2+L_3$ affects the interior equations.  We absorb $L_3$ into
$L_2$ and take $L_3=0$.  Set $\mathbf Q_i=(\mathbf Q_{i1},\mathbf Q_{i2},\mathbf Q_{i3})$.  The curl--div
identity yields
\[
\begin{aligned}
&\int_\Omega\left(L_1\partial_k\mathbf Q_{ij}\partial_k\mathbf Q_{ij}+L_2\partial_j\mathbf Q_{ij}\partial_k\mathbf Q_{ik}\right)\mathrm{d} x
=\sum_{i=1}^3\int_\Omega\left(L_1|\nabla \mathbf Q_i|^2+L_2|\nabla\cdot \mathbf Q_i|^2\right)\mathrm{d} x\\
&\quad=\sum_{i=1}^3\int_\Omega\left(L_1|\nabla\times \mathbf Q_i|^2+(L_1+L_2)|\nabla\cdot \mathbf Q_i|^2\right)\mathrm{d} x+\text{boundary terms}\\
&\quad\geq\min\{L_1,L_1+L_2\}\int_\Omega|\nabla \mathbf Q|^2\mathrm{d} x+\text{boundary terms}.
\end{aligned}
\]
Thus, when the boundary contribution vanishes or is fixed, the integrated
elastic energy is non-negative under the conditions
\[
L_1\geq0,\quad L_1+L_2\geq0.
\]
After scaling by $L_1>0$, we set $L_1=1$ and $L_2=L$.  The admissible
three-dimensional range is then $L\geq-1$.

Biaxiality is relevant near an interface because the eigenvalue structure of
$\mathbf Q$ may change across the transition layer \cite{Popa1997}.  We describe this
structure through the decomposition:
\[ 
\mathbf Q=\lambda\left(\mathbf m\otimes \mathbf m-\frac{1}{3}\mathbf I\right)+\mu\left(\mathbf n\otimes \mathbf n-\frac{1}{3}\mathbf I\right),
\]
where $\lambda,\mu\in\R$ and $\mathbf m,\mathbf n\in\mathbb S^2$ may be
chosen orthogonal.  A state is biaxial if $\mathbf Q$
has three distinct eigenvalues. Instead, a uniaxial state has the form
\[
\mathbf Q=s\left(\mathbf n\otimes \mathbf n-\frac{1}{3}\mathbf I\right),\quad s\in\R,\quad \mathbf n\in\mathbb S^2,
\]
whose bulk energy reduces to
\[
f_B^{a,b,c}(\mathbf Q)=\frac{s^2}{27}\left(9a-2bs+3cs^2\right).
\]
For $a,b,c>0$, coexistence means $b^2=27ac$.  The spectral inequality
\cite{MZ10}
\[
\tr(\mathbf Q^3)\leq\frac{1}{\sqrt6}|\mathbf Q|^3,
\quad |\mathbf Q|^2=\tr(\mathbf Q^2),
\quad \mathbf Q\in\Szero,
\]
gives
\[
f_B^{a,b,c}(\mathbf Q)
\geq\frac{c}{4}|\mathbf Q|^2
\left(|\mathbf Q|-\sqrt{\frac{2a}{c}}\right)^2\geq0.
\]
Equality in the spectral inequality and in the factorization gives the well
set
\begin{equation}\label{eq:bulk-wells}
\{0\}\cup\left\{\sqrt{\frac{3a}{c}}
\left(\mathbf n\otimes \mathbf n-\frac{1}{3}\mathbf I\right)\colon
\mathbf n\in\mathbb S^2\right\}.
\end{equation}
We adopt the non-dimensional normalization
\[
a=1,\quad b=9,\quad c=3.
\]
The interface joins the isotropic well at $z=-\infty$ with the nematic well at
$z=+\infty$. We impose:
\begin{equation}\label{eq:intro-interface-bc}
\mathbf Q(x_1,x_2,-\infty)=0,\quad \mathbf Q(x_1,x_2,+\infty)=\mathbf n(x_1,x_2)\otimes \mathbf n(x_1,x_2)-\frac{1}{3}\mathbf I.
\end{equation}
A one-dimensional profile $X\colon\R\to\R^m$ is called a heteroclinic
connection from $X_-$ to $X_+$ if
\[
X_-\neq X_+,\quad
\lim_{z\to-\infty}X(z)=X_-,\quad
\lim_{z\to+\infty}X(z)=X_+.
\]
Thus, an isotropic--nematic interface is a heteroclinic connection in the
order-parameter space, with $z$ as the orbit parameter.
Let $\boldsymbol{\nu}$ be the unit normal to the interface.  Homeotropic anchoring means
that $\mathbf n$ is parallel to $\boldsymbol{\nu}$.  Planar anchoring means that $\mathbf n$ is orthogonal
to $\boldsymbol{\nu}$.  Intermediate orientations are tilted.  For a profile that
depends only on the normal coordinate $z$, the reduced energy is
\begin{equation}\label{eq:intro-one-dimensional-energy}
\mathcal F_0(\mathbf Q)=\int_\R\left\{\frac{1}{2}|\mathbf Q'|^2+\frac{L}{2}\sum_{i=1}^3((\mathbf Q')_{i3})^2+f_B(\mathbf Q)\right\}\mathrm{d} z,
\end{equation}
Here $f_B=f_B^{1,9,3}$, and $(\cdot)'=\partial_z(\cdot)$.

The sign of $L$ selects the preferred anchoring.  The standard comparison
between homeotropic and planar uniaxial profiles goes back to
\cite{DeGennes1971}.  In the normalization of
\eqref{eq:intro-one-dimensional-energy}, direct substitution gives the
respective elastic densities
\[
\frac{1}{2}\left(1+\frac{2}{3}L\right)|\partial_z\mathbf Q|^2,\quad
\frac{1}{2}\left(1+\frac{1}{6}L\right)|\partial_z\mathbf Q|^2.
\]
The first coefficient is smaller when $L<0$, while the second is smaller
when $L>0$.  This comparison favors homeotropic anchoring for negative $L$
and planar anchoring for positive $L$.

The case $L=0$ admits an explicit description.  After translating the result
of Park, Wang, Zhang, and Zhang \cite{Park2017} into the normalization used
here, a global one-dimensional minimizer subject to
\eqref{eq:intro-interface-bc} has the form
\[
\mathbf Q(z)=\frac{1}{2}\left(1+\tanh\frac{z-z_0}{2}\right)
\left(\mathbf n\otimes \mathbf n-\frac{1}{3}\mathbf I\right),
\]
where $z_0\in\R$.  This is the profile in the normalization used here.  The
spatial variable must be rescaled under a different bulk normalization.
Park, Wang, Zhang, and Zhang also constructed a homeotropic equilibrium $\mathbf Q_h$ for
$L\neq0$.  This equilibrium is stable under one-dimensional perturbations
for $-\frac{3}{2}<L\leq0$ and unstable for $L>0$.  Chen, Zhang, and Zhang
\cite{Chen2018} proved a related rigidity result for the full energy at $L=0$.  Their result requires assumptions about the level sets.

Wu proved that the homeotropic profile $\mathbf Q_h$ is stable under
three-dimensional perturbations for $-1\leq L\leq0$
\cite[Theorem 3.1]{Wu2025}.  In addition, for $L>-\frac{3}{2}$,
\cite[Theorem 3.2]{Wu2025} constructs a non-negative minimizer of the reduced
energy within the diagonal planar class.

We write this
profile as $\mathbf Q_p=\mathbf Q^d(S,T)$, where
\begin{equation}\label{planar}
\mathbf Q^d(S,T):=\frac{1}{3}\diag(S+3T,S-3T,-2S).
\end{equation}
Substitution into \eqref{eq:intro-one-dimensional-energy} gives the reduced
energy
\begin{equation}\label{eq:reduced-energy}
\begin{aligned}
\mathcal E_L(S,T)&:=3\mathcal F_0(\mathbf Q^d(S,T))
=\int_\R\{\alpha(S')^2+3(T')^2+W(S,T)\}\mathrm{d} z,
\end{aligned}
\end{equation}
where
\[
\alpha:=1+\frac{2L}{3},\quad
W(S,T):=S^2+3T^2+2(S^3-9ST^2)+(S^2+3T^2)^2.
\]
The boundary conditions for the diagonal planar profile \eqref{planar} are
\begin{equation}\label{eq:planar-bc}
(S,T)(-\infty)=(0,0),\quad
(S,T)(+\infty)=\left(\frac{1}{2},\frac{1}{2}\right).
\end{equation}
Define the diagonal planar admissible class
\[
\mathcal A_{\mathrm p}
:=\{(S,T)\in H^1_{\mathrm{loc}}(\R;\R^2)\colon
\mathcal E_L(S,T)<+\infty,\quad
(S,T)\text{ satisfies \eqref{eq:planar-bc}}\}.
\]
We call $(S,T)\in\mathcal A_{\mathrm p}$ a diagonal planar minimizer if
\begin{equation}\label{eq:diagonal-minimizer}
\mathcal E_L(S,T)
=\inf_{(\widetilde S,\widetilde T)\in\mathcal A_{\mathrm p}}
\mathcal E_L(\widetilde S,\widetilde T).
\end{equation}
This terminology refers only to minimization after imposing the diagonal
ansatz \eqref{planar}; it does not assert that $\mathbf Q^d(S,T)$ minimizes
\eqref{eq:intro-one-dimensional-energy} among all $\Szero$-valued profiles
satisfying \eqref{eq:intro-interface-bc}.

The Euler--Lagrange equations associated with \eqref{eq:reduced-energy} are
\begin{equation}\label{eq:EL}
\begin{cases}
\alpha S''=S+3S^2-9T^2+2S(S^2+3T^2),\\
T''=T-6ST+2T(S^2+3T^2).
\end{cases}
\end{equation}
The condition $\alpha>0$, or equivalently $L>-\frac{3}{2}$, is the coercivity
condition for \eqref{eq:reduced-energy}.

For a profile $\mathbf Q$ at which the first variation of
\eqref{eq:intro-one-dimensional-energy} vanishes, called a critical profile,
we use the normalized second variation
\begin{equation}\label{eq:second-variation-definition}
E_{\mathbf Q}(\mathbf P,\mathbf P)
:=\lim_{\varepsilon\to0}
\frac{1}
{3\varepsilon^2}(\mathcal F_0(\mathbf Q+\varepsilon\mathbf P)-\mathcal F_0(\mathbf Q)).
\end{equation}
Stability and instability under one-dimensional perturbations mean,
respectively,
\[
\begin{aligned}
\mathbf Q\ \text{is stable}
&\quad\Longleftrightarrow\quad
E_{\mathbf Q}(\mathbf P,\mathbf P)\geq0
\quad\text{for any }\mathbf P\in C_0^\infty(\R;\Szero),\\
\mathbf Q\ \text{is unstable}
&\quad\Longleftrightarrow\quad
E_{\mathbf Q}(\mathbf P,\mathbf P)<0
\quad\text{for some }\mathbf P\in C_0^\infty(\R;\Szero).
\end{aligned}
\]
For a planar profile $(S,T)$ as in \eqref{planar}, define
\begin{equation}\label{eq:VL}
a_L:=\frac{2+L}{2},\quad
V_L:=1+3(S-3T)+2(S^2+3T^2),
\end{equation}
\begin{equation}\label{eq:critical-form}
\mathfrak q_L[p]
:=\int_\R(a_L(p')^2+V_Lp^2)\mathrm{d} z,
\quad p\in H^1(\R),
\end{equation}
and the associated self-adjoint operator
\begin{equation}\label{eq:critical-operator}
\mathscr L_L:=-a_L\frac{\mathrm{d}^2}{\mathrm{d} z^2}+V_L(z),\quad
\mathcal D(\mathscr L_L)=H^2(\R)\subset L^2(\R).
\end{equation}
For $L>-\frac{3}{2}$, \eqref{eq:VL} gives $a_L>\frac{1}{4}$.  By
\eqref{eq:planar-bc}, there exists $R>0$ such that
\[
|S(z)|+|T(z)|\leq2,\quad|z|\geq R.
\]
Continuity on $[-R,R]$ gives
\[
\sup_{z\in\R}(|S(z)|+|T(z)|)
\leq\max\left\{2,\max_{|z|\leq R}(|S(z)|+|T(z)|)\right\}<+\infty.
\]
Therefore, $V_L\in L^\infty(\R)$ by \eqref{eq:VL}, and
\eqref{eq:critical-operator} is self-adjoint on $H^2(\R)$ as a bounded
perturbation of $-a_L\frac{\mathrm{d}^2}{\mathrm{d} z^2}$ by the
Kato--Rellich theorem \cite[Theorem 6.4]{Teschl2014}.

For the off-diagonal perturbation \eqref{eq:Pp}, the tensorial second
variation reduces to \eqref{eq:critical-form} by \eqref{eq:I31}.  The
instability proof for $L<0$ in \cite[Theorem 3.3]{Wu2025} assumes the
additional condition
\begin{equation}\label{eq:old-condition}
\int_{\{T>S\}}T^2((T'-S')^2-9(T-S)^3)\mathrm{d} z<0.
\end{equation}

Our aim is to solve this planar stability problem without this assumption.
We first prove new derivative inequalities for the diagonal
planar minimizer.  These inequalities reveal a rotation mode that produces a negative
second variation.  They also give a sharp instability range for negative
$L$.  We then analyze the threshold $L=0$ and define an optimal stability
index as the bottom of a scalar spectrum.  Finally, we correct the
normalization of a previously stated explicit interface
profile. Removing \eqref{eq:old-condition} is, therefore, one part of a
broader analysis.

The negative direction has a geometric interpretation.  For planar anchoring, the director is
tangent to the interface and orthogonal to its normal.  When $L<0$, the
anisotropic elastic term favors alignment with the normal instead.  A small
rotation toward the normal lowers this term near the interface.  The
derivative inequalities proved below show that the second
variation is strictly negative.  This is consistent with the oblique or
tilted anchoring described in \cite{Kamil2009a}. It also agrees with the strong homeotropic anchoring in the dynamic sharp-interface limit of \cite{Dong2025}.

\begin{theorem}[Sharp planar instability]\label{thm:main}
Let $-\frac{3}{2}<L<0$, and let $(S,T)\in\mathcal A_{\mathrm p}$ be a
non-negative diagonal planar minimizer. Let
\[
\mathbf Q_p(z)=\frac{1}{3}\diag(S+3T,S-3T,-2S).
\]
Then there exists $\mathbf P\in C_0^\infty(\R;\Szero)$ such that the normalized
second variation \eqref{eq:second-variation-definition} satisfies
\[
E_{\mathbf Q_p}(\mathbf P,\mathbf P)<0.
\]
In particular, $\mathbf Q_p$ is unstable with respect to one-dimensional
perturbations and hence also unstable in any larger perturbation class
that contains them.
\end{theorem}

\begin{rem}
Several comments are in order.
\begin{enumerate}
\item We first outline the proof.  Proposition~\ref{prop:monotone} and
Lemma~\ref{lem:derivative-order} establish profile estimates
\[
S'>0,\quad T'>0,\quad 3T'-S'>0\quad\hbox{on }\R.
\]
The third inequality is the key new point.

Lemma~\ref{lem:rotation} shows
that the critical off-diagonal operator $\mathscr L_L$ defined in
\eqref{eq:critical-operator} satisfies
\[
\mathscr L_L(S+T)=\frac{L}{6}(S''-3T'').
\]
The same lemma also gives the absolutely convergent identity
\[
\mathfrak q_L[S+T]
=\frac{L}{6}\int_\R(S'+T')(3T'-S')\mathrm{d} z<0.
\]

Note that the function $S+T$ is not in $L^2(\R)$ as it tends to $1$ as
$z\to+\infty$. A slow cutoff in the nematic tail adds arbitrarily little
energy and produces a compactly supported negative direction.  This argument
turns the geometric rotation of the director into a strict second-variation
test.
\item The range in Theorem~\ref{thm:main} is sharp for the reduced planar
problem in the following sense.  The coefficient
$\alpha=1+\frac{2L}{3}$ in \eqref{eq:reduced-energy} is exactly positive for 
$L>-\frac{3}{2}$, and Theorem~\ref{thm:main} covers its entire negative
range $(-\frac{3}{2},0)$.  The full three-dimensional elastic energy is
non-negative only for $L\geq-1$ and strictly coercive for $L>-1$. Hence, its
negative admissible range is $[-1,0)$, while $(-\frac{3}{2},-1)$ belongs only to the coercive planar reduction.  

The optimal second-variation value at fixed $L$ is
\begin{equation}\label{eq:Lambda}
\Lambda(L):=\inf\sigma(\mathscr L_L)
=\inf_{p\in H^1(\R)\backslash\{0\}}
\frac{\mathfrak q_L[p]}{\|p\|_{L^2}^2}.
\end{equation}
We prove $\Lambda(L)<0$ for $-\frac{3}{2}<L<0$ and
$\Lambda(0)=0$ in Proposition~\ref{prop:zero-resonance} by the exact
factorization \eqref{eq:L0-factorization}. Thus, $L=0$ is the endpoint of
this negative spectral mechanism.  No stability conclusion for $L>0$ is claimed.
\end{enumerate}
\end{rem}

\subsection{Organization of this paper} Section~\ref{sec:model} fixes the normalization and the exact length scales of the explicit profiles.  Section~\ref{sec:profile} proves the estimates for the diagonal planar minimizer.  Section~\ref{sec:spectral} identifies
the scalar second-variation operator, and Section~\ref{sec:instability}
proves Theorem~\ref{thm:main}. The endpoint factorization is given in
Section~\ref{sec:threshold}.

\section{Reduced equations and explicit profiles}\label{sec:model}

Under the normalization $a=1$, $b=9$, $c=3$, $L_1=1$, and $L_2=L$, we recall that the
three-dimensional energy is
\[
\mathcal F(\mathbf Q)=\int_{\R^3}\left\{
f_B(\mathbf Q)+\frac{1}{2}\partial_k\mathbf Q_{ij}\partial_k\mathbf Q_{ij}
+\frac{L}{2}\partial_j\mathbf Q_{ij}\partial_k\mathbf Q_{ik}\right\}\mathrm{d} x,
\]
where
\[
f_B(\mathbf Q)=\frac{1}{2}\tr(\mathbf Q^2)-3\tr(\mathbf Q^3)
+\frac{3}{4}(\tr(\mathbf Q^2))^2.
\]
For $\mathbf Q=\mathbf Q(z)$, this energy reduces to
\eqref{eq:intro-one-dimensional-energy}.  Substituting
$\mathbf Q=\mathbf Q^d(S,T)$ gives \eqref{eq:reduced-energy}, and its first
variation gives the equations \eqref{eq:EL} with boundary conditions
\eqref{eq:planar-bc}.

The length scale is fixed by \eqref{eq:EL}.  For $L>-\frac{3}{2}$,
\eqref{eq:reduced-energy} gives $\alpha>0$.  The homeotropic condition
$\mathbf n\parallel\boldsymbol{\nu}$ corresponds in the diagonal
ansatz \eqref{planar} to $S_h=-r$ and $T_h\equiv0$, where $r$ solves
the scalar boundary-value problem
\begin{equation}\label{eq:homeotropic-bvp}
\begin{gathered}
\alpha r''=r(1-r)(1-2r),\quad z\in\R,\\
r(-\infty)=0,\quad r(+\infty)=1.
\end{gathered}
\end{equation}
This is the homeotropic equilibrium considered in \cite{Park2017}.  By
\eqref{eq:homeotropic-bvp},
\begin{equation}\label{eq:homeotropic-first-integral}
\frac{\mathrm{d}}{\mathrm{d} z}
\left\{\frac{\alpha}{2}(r')^2-\frac{1}{2}r^2(1-r)^2\right\}
=r'\bigl\{\alpha r''-r(1-r)(1-2r)\bigr\}=0.
\end{equation}
By the endpoint conditions in \eqref{eq:homeotropic-bvp} and the mean-value
theorem, there exist $z_n^-\in(-n-1,-n)$ and $z_n^+\in(n,n+1)$ such that
\[
r'(z_n^-)=r(-n)-r(-n-1)\to0,\quad
r'(z_n^+)=r(n+1)-r(n)\to0.
\]
Evaluating \eqref{eq:homeotropic-first-integral} along either sequence shows
that its constant is zero.  For the
increasing connection, 
\[
\sqrt\alpha\,r'=r(1-r),\quad
\log\frac{r}{1-r}=\frac{z-z_0}{\sqrt\alpha}.
\]
Hence
\[
S_h(z)=-\frac{1}{2}\left(1+
\tanh\frac{z-z_0}{2\sqrt\alpha}\right),\quad
T_h(z)=0,
\]
and these functions satisfy the homeotropic boundary conditions
\[
S_h(-\infty)=0,\quad S_h(+\infty)=-1,\quad T_h\equiv0.
\]

At $L=0$, the substitution $S=T=s$ in
\eqref{eq:EL} with boundary conditions \eqref{eq:planar-bc} gives the scalar
boundary-value problem
\begin{equation}\label{eq:planar-L0-bvp}
\begin{gathered}
s''=s(1-2s)(1-4s),\quad z\in\R,\\
s(-\infty)=0,\quad s(+\infty)=\frac{1}{2}.
\end{gathered}
\end{equation}
The increasing solution is
\begin{equation}\label{eq:correct-planar-L0}
s(z)=\frac{1}{4}\left(1+\tanh\frac{z-z_0}{2}\right).
\end{equation}
Indeed, \eqref{eq:correct-planar-L0} satisfies
\[
s'=s(1-2s),\quad
s''=s(1-2s)(1-4s),\quad
s(-\infty)=0,\quad s(+\infty)=\frac{1}{2},
\]
and therefore solves \eqref{eq:planar-L0-bvp}.  The corresponding planar
profile is
\[
\mathbf Q_p(z)=\mathbf Q^d(s(z),s(z)),
\]
which is the $L=0$ interface described in \cite{Park2017}.

\section{The diagonal planar minimizer}\label{sec:profile}

A pair satisfying \eqref{eq:diagonal-minimizer} is constructed in
\cite[Theorem 3.2]{Wu2025}; the minimization is over
$\mathcal A_{\mathrm p}$, not over all $\Szero$-valued profiles.  Its proof
applies the degenerate-metric method of \cite[Theorem 3.1]{Zuniga2016}.  We
record the properties of this diagonal planar minimizer used below.

\begin{proposition}\label{prop:existence}
Let $\alpha>0$.  The energy \eqref{eq:reduced-energy} has a minimizer
$(S,T)\in\mathcal A_{\mathrm p}$ for which $S,T\geq0$.  Any such
non-negative diagonal planar minimizer is smooth,
and there exist $C,c>0$ such that
\begin{equation}\label{eq:profile-decay}
\begin{aligned}
|S(z)|+|T(z)|
+\sum_{j=1}^2(|S^{(j)}(z)|+|T^{(j)}(z)|)
&\leq C\exp(cz),\quad z\leq0,\\
\left|S(z)-\frac{1}{2}\right|+\left|T(z)-\frac{1}{2}\right|
+\sum_{j=1}^2(|S^{(j)}(z)|+|T^{(j)}(z)|)
&\leq C\exp(-cz),\quad z\geq0.
\end{aligned}
\end{equation}
Moreover,
\begin{equation}\label{eq:box-bounds}
0<S(z)\leq\frac{1}{2},\quad 0<T(z)\leq\frac{1}{2},\quad z\in\R.
\end{equation}
\end{proposition}

\begin{proof}
The existence with $S,T\geq0$ follows from
\cite[Theorem 3.2]{Wu2025}.  For any such diagonal planar minimizer,
\cite[Proposition 6.1 and its proof]{Wu2025} gives smoothness, the terms in
\eqref{eq:profile-decay} of order at most one, and
\[
0\leq S(z)\leq\frac{1}{2},\quad
0\leq T(z)\leq\frac{1}{2},\quad z\in\R,
\]
where the limiting values are fixed by \eqref{eq:planar-bc}.  Substitution
of these estimates into \eqref{eq:EL} gives the second-derivative terms in
\eqref{eq:profile-decay}.  It remains to prove strict positivity.
Since $\alpha>0$, the right-hand side of the first-order system associated
with \eqref{eq:EL} is polynomial and hence locally Lipschitz; its Cauchy
problem is unique.

If $T(z_0)=0$, then $T'(z_0)=0$ because $T\geq0$.  The second equation in
\eqref{eq:EL} has the form
\[
T''=(1-6S+2S^2+6T^2)T.
\]
The uniqueness of this scalar equation with $T(z_0)=T'(z_0)=0$ gives
$T\equiv0$, contrary to \eqref{eq:planar-bc}. 

If $S(z_0)=0$, then
$S'(z_0)=0$ and $S''(z_0)\geq0$.  The first equation in \eqref{eq:EL}
gives
\[
0\leq\alpha S''(z_0)=-9T(z_0)^2\leq0.
\]
Thus $T(z_0)=T'(z_0)=0$.  The Cauchy datum for \eqref{eq:EL} is then
\[
(S,T,S',T')(z_0)=(0,0,0,0),
\]
so uniqueness gives the constant solution, again contradicting
\eqref{eq:planar-bc}. 

Hence $S,T>0$ on $\R$, which proves \eqref{eq:box-bounds}.
\end{proof}

The first integral of \eqref{eq:EL} will be used repeatedly.

\begin{lemma}\label{lem:equipartition}
Let $\alpha>0$.  Any solution of \eqref{eq:EL} satisfying
\eqref{eq:planar-bc} and
\begin{equation}\label{finiteenergyass}
\mathcal E_L(S,T)<+\infty,
\end{equation}
where $\mathcal E_L$ is defined in \eqref{eq:reduced-energy},
satisfies
\begin{equation}\label{eq:equipartition}
\alpha(S')^2+3(T')^2=W(S,T).
\end{equation}
\end{lemma}

\begin{proof}
Set
\[
H:=\alpha(S')^2+3(T')^2-W(S,T).
\]
Multiply the two equations in \eqref{eq:EL} by $2S'$ and $6T'$,
respectively.  The result is
\[
H'=0.
\]
Since $W=3f_B(\mathbf Q^d(S,T))\geq0$ by \eqref{eq:bulk-wells},
the finite-energy assumption \eqref{finiteenergyass} and \eqref{eq:reduced-energy} give
\[
S',T'\in L^2(\R).
\]
Consequently,
\[
\int_{-n-1}^{-n}\bigl((S')^2+(T')^2\bigr)\mathrm{d} z\to0,
\quad n\to+\infty.
\]
Choose $z_n\in[-n-1,-n]$ such that
\[
(S'(z_n))^2+(T'(z_n))^2
\leq\int_{-n-1}^{-n}\bigl((S')^2+(T')^2\bigr)\mathrm{d} z.
\]
Then $z_n\to-\infty$ and
\[
S'(z_n)\to0,\quad T'(z_n)\to0.
\]
Equation~\eqref{eq:planar-bc} gives
\[
W(S(z_n),T(z_n))\to W(0,0)=0.
\]
Thus $H=H(z_n)\to0$.  Since $H'=0$, we have $H\equiv0$, which is
\eqref{eq:equipartition}.
\end{proof}

The derivative signs required in Section~\ref{sec:instability} follow from
minimality within $\mathcal A_{\mathrm p}$ as follows.

\begin{proposition}\label{prop:monotone}
Any non-negative diagonal planar minimizer in
Proposition~\ref{prop:existence} satisfies
\[
S'(z)>0,\quad T'(z)>0,\quad z\in\R.
\]
\end{proposition}

\begin{proof}
Set $u=S'$ and $v=T'$, and define
\begin{equation}\label{eq:Jacobi-coefficients}
\begin{aligned}
A&=1+6S+6S^2+6T^2,\quad B=6T(2S-3),\\
C&=2T(2S-3),\quad D=1-6S+2S^2+18T^2.
\end{aligned}
\end{equation}
Differentiating \eqref{eq:EL} gives the Jacobi system
\begin{equation}\label{eq:Jacobi}
\alpha u''=Au+Bv,\quad v''=Cu+Dv,
\end{equation}
where $B=3C<0$ by \eqref{eq:box-bounds}.  By
\eqref{eq:reduced-energy}, the second variation in the diagonal direction
$(\xi,\eta)$ is
\begin{equation}\label{eq:diagonal-Jacobi-form}
\begin{aligned}
\mathcal J[\xi,\eta]
&:=\frac{1}{2}\left.\frac{\mathrm{d}^2}{\mathrm{d}\varepsilon^2}
\mathcal E_L(S+\varepsilon\xi,T+\varepsilon\eta)
\right|_{\varepsilon=0}\\
&=\int_\R
\{\alpha(\xi')^2+3(\eta')^2+A\xi^2+2B\xi\eta+3D\eta^2\}\mathrm{d} z.
\end{aligned}
\end{equation}
For $(\xi,\eta)\in C_0^\infty(\R)\times C_0^\infty(\R)$ and sufficiently
small $|\varepsilon|$,
\[
(S+\varepsilon\xi,T+\varepsilon\eta)\in\mathcal A_{\mathrm p}.
\]
Hence \eqref{eq:diagonal-minimizer} gives
\[
\mathcal J[\xi,\eta]\geq0,
\quad (\xi,\eta)\in C_0^\infty(\R)\times C_0^\infty(\R).
\]

Define
\[
u_+:=\max\{u,0\},\quad u_-:=\max\{-u,0\},\quad
v_+:=\max\{v,0\},\quad v_-:=\max\{-v,0\}.
\]
By \eqref{eq:profile-decay}, $u,v\in H^1(\R)$.  Since the positive- and
negative-part maps are Lipschitz,
\[
(u_+)'=\boldsymbol{1}_{\{u>0\}}u',\quad
(u_-)'=-\boldsymbol{1}_{\{u<0\}}u',
\]
and the analogous identities for $v_\pm$.  Hence
\[
u_\pm,v_\pm\in H^1(\R).
\]
If
$\zeta_R\in C_0^\infty(\R)$ satisfies
\[
0\leq\zeta_R\leq1,\quad \zeta_R=1\ \text{on }[-R,R],\quad
|\zeta_R'|\leq\frac{C}{R},
\]
then \eqref{eq:profile-decay} gives
\[
\mathcal J[\zeta_Ru_+,\zeta_Rv_+]
\to\mathcal J[u_+,v_+],\quad R\to+\infty.
\]
Here the potential terms converge by \eqref{eq:Jacobi-coefficients} and
\eqref{eq:box-bounds}, which give $A,B,D\in L^\infty(\R)$.
Consequently $\mathcal J[u_+,v_+]\geq0$.

Using \eqref{eq:profile-decay} to justify integration by parts, test the two
equations in \eqref{eq:Jacobi} by $u_+$ and $3v_+$ and compare with
\eqref{eq:diagonal-Jacobi-form} yields
\begin{equation}\label{eq:positive-parts}
\mathcal J[u_+,v_+]
=-\int_\R B(vu_++uv_+-2u_+v_+)\mathrm{d} z.
\end{equation}
By the definitions of $u_\pm$ and $v_\pm$, the bracket in
\eqref{eq:positive-parts} satisfies
\[
vu_++uv_+-2u_+v_+=-u_+v_--u_-v_+\leq0.
\]
Since $B<0$, \eqref{eq:positive-parts} and
$\mathcal J[u_+,v_+]\geq0$ imply
\[
0\leq\mathcal J[u_+,v_+]
=\int_\R B(u_+v_-+u_-v_+)\mathrm{d} z\leq0.
\]
Because $B<0$ at any finite point by \eqref{eq:box-bounds},
$u_+v_-=u_-v_+=0$ follows from continuity, and hence
\begin{equation}\label{eq:uv-positive}
uv\geq0,\quad z\in\R.
\end{equation}

Let
\[
(f,g,\gamma,\beta,\rho)
\in\{(u,v,\alpha,B,A),(v,u,1,C,D)\}.
\]
Equations \eqref{eq:Jacobi}, \eqref{eq:Jacobi-coefficients},
\eqref{eq:box-bounds}, and \eqref{eq:uv-positive} give
\[
\gamma f''=\rho f+\beta g,
\quad fg\geq0,
\quad \gamma>0,
\quad \beta<0.
\]
Suppose that $f$ changes sign.  Replacing $(f,g)$ by $(-f,-g)$ if
necessary, choose $a<b$ such that $f(a)<0<f(b)$, and define
\[
z_-:=\sup\{z\in[a,b]\colon f(z)<0\},\quad
z_+:=\inf\{z\in[z_-,b]\colon f(z)>0\}.
\]
Then
\[
f=0\quad\text{on }[z_-,z_+].
\]
If $z_-=z_+=z_*$, there exist $x_n,y_n\to z_*$ such that
\[
f(x_n)<0<f(y_n),\quad
g(x_n)\leq0\leq g(y_n),
\]
where the second pair of inequalities follows from $fg\geq0$.  Hence
$g(z_*)=0$ by continuity.  If $z_-<z_+$, then, for any
$z_*\in(z_-,z_+)$,
\[
0=\gamma f''(z_*)=\beta(z_*)g(z_*),
\quad \beta(z_*)<0,
\]
so $g(z_*)=0$.  Therefore,
\begin{equation}\label{eq:common-zero}
u\ \text{or}\ v\ \text{changes sign}
\quad\Longrightarrow\quad
u(z_*)=v(z_*)=0\quad\text{for some }z_*\in\R.
\end{equation}
Assume the left-hand side of \eqref{eq:common-zero}.  Lemma~\ref{lem:equipartition}
and \eqref{eq:equipartition} give
$W(S(z_*),T(z_*))=0$.  Since
$W=3f_B(\mathbf Q^d(S,T))$, \eqref{eq:bulk-wells} gives
\[
\mathbf n\otimes\mathbf n-\frac{1}{3}\mathbf I\ \text{diagonal}
\quad\Longleftrightarrow\quad
n_in_j=0\ \text{for }i\neq j
\quad\Longleftrightarrow\quad
\mathbf n\in\{\pm\mathbf e_1,\pm\mathbf e_2,\pm\mathbf e_3\}.
\]
Using \eqref{planar},
\[
\begin{aligned}
\mathbf Q^d(S,T)=0
&\quad\Longleftrightarrow\quad(S,T)=(0,0),\\
\mathbf Q^d(S,T)=\mathbf e_1\otimes\mathbf e_1-\frac{1}{3}\mathbf I
&\quad\Longleftrightarrow\quad
(S,T)=\left(\frac{1}{2},\frac{1}{2}\right),\\
\mathbf Q^d(S,T)=\mathbf e_2\otimes\mathbf e_2-\frac{1}{3}\mathbf I
&\quad\Longleftrightarrow\quad
(S,T)=\left(\frac{1}{2},-\frac{1}{2}\right),\\
\mathbf Q^d(S,T)=\mathbf e_3\otimes\mathbf e_3-\frac{1}{3}\mathbf I
&\quad\Longleftrightarrow\quad(S,T)=(-1,0).
\end{aligned}
\]
Therefore,
\[
\{W=0\}\cap[0,+\infty)^2
=\left\{(0,0),\left(\frac{1}{2},\frac{1}{2}\right)\right\},
\]
and
\begin{equation}\label{eq:equilibrium-cauchy-data}
(S,T,S',T')(z_*)
\in\left\{(0,0,0,0),
\left(\frac{1}{2},\frac{1}{2},0,0\right)\right\}.
\end{equation}
For $\mathbf Y=(S,T,u,v)$, equation \eqref{eq:EL} is
\[
\mathbf Y'=\mathbf F(\mathbf Y),
\]
where
\[
\mathbf F(S,T,u,v)
=\left(
u,v,
\frac{S+3S^2-9T^2+2S(S^2+3T^2)}{\alpha},
T-6ST+2T(S^2+3T^2)
\right).
\]
Since $\alpha>0$, $\mathbf F\in C^\infty(\R^4;\R^4)$ is locally
Lipschitz.  Direct substitution into the formula for $\mathbf F$ gives
\[
\mathbf F(0,0,0,0)=0,\quad
\mathbf F\left(\frac{1}{2},\frac{1}{2},0,0\right)=0.
\]
Hence the two constant functions with values in
\eqref{eq:equilibrium-cauchy-data} solve \eqref{eq:EL}.  The uniqueness of
the Cauchy problem with the initial data \eqref{eq:equilibrium-cauchy-data}
gives
\[
\mathbf Y(z)\equiv(0,0,0,0)
\quad\text{or}\quad
\mathbf Y(z)\equiv\left(\frac{1}{2},\frac{1}{2},0,0\right),
\quad z\in\R,
\]
contradicting \eqref{eq:planar-bc}. Thus, $u$ and $v$ cannot
vanish simultaneously, and \eqref{eq:common-zero} shows that neither
function changes sign.  Moreover,
\eqref{eq:profile-decay} and \eqref{eq:planar-bc} give
\begin{align*}
\int_\R u\,\mathrm{d} z
&=S(+\infty)-S(-\infty)=\frac{1}{2},\\
\int_\R v\,\mathrm{d} z
&=T(+\infty)-T(-\infty)=\frac{1}{2}.
\end{align*}
Consequently $u,v\geq0$.

Finally, if $u(z_0)=0$, then $u''(z_0)\geq0$, whereas the first equation in
\eqref{eq:Jacobi} gives $\alpha u''(z_0)=B(z_0)v(z_0)\leq0$.  Equality would
force $v(z_0)=0$, already excluded.  Thus $u>0$.  If $v(z_0)=0$, then
$v''(z_0)\geq0$, whereas the second equation in \eqref{eq:Jacobi} gives
$v''(z_0)=C(z_0)u(z_0)<0$.  Hence $v>0$ as well.
\end{proof}

The following upper bound on $T-S$ is used in
Lemma~\ref{lem:derivative-order}.

\begin{lemma}\label{lem:gap}
Suppose $0<\alpha<1$.  Then any non-negative diagonal planar minimizer in
Proposition~\ref{prop:existence} satisfies
\begin{equation}\label{eq:gap-bound}
T(z)-S(z)<\frac{1}{6},\quad z\in\R.
\end{equation}
\end{lemma}

\begin{proof}
Put $\phi=T-S$.  A direct subtraction of the equations in
\eqref{eq:EL} gives
\begin{equation}\label{eq:phi-equation}
\phi''-c\phi=(\alpha-1)S'',
\quad c=1+3S+9T+2S^2+6T^2.
\end{equation}
Equation~\eqref{eq:box-bounds} gives $c>0$.
If $\sup\phi\leq0$ there is nothing to prove.  Otherwise the positive
maximum is attained at some finite $z_m$.  Indeed, \eqref{eq:planar-bc}
gives
\[
\lim_{z\to-\infty}\phi=0,\quad
\lim_{z\to+\infty}\phi=0.
\]
At $z_m$, $\phi''(z_m)\leq0$, so \eqref{eq:phi-equation} implies
\[
(1-\alpha)S''(z_m)\geq c(z_m)\phi(z_m)>0.
\]
Thus $S''(z_m)>0$.  By the first equation in \eqref{eq:EL},
\[
0<\alpha S''(z_m)
=S(z_m)+3S(z_m)^2+2S(z_m)^3
-(9-6S(z_m))T(z_m)^2,
\]
and \eqref{eq:box-bounds} therefore gives
\[
T(z_m)<T_s(S(z_m)),\quad
T_s(r)=\sqrt{\frac{r+3r^2+2r^3}{9-6r}}.
\]
For $0\leq r\leq\frac{1}{2}$ one has $T_s(r)<r+\frac{1}{6}$.  Indeed,
\begin{align*}
&(9-6r)\left(r+\frac{1}{6}\right)^2
-(r+3r^2+2r^3)\\
&\quad=4r^2(1-2r)+\frac{11}{6}r+\frac{1}{4}>0.
\end{align*}
Therefore $\phi(z_m)<\frac{1}{6}$, which proves \eqref{eq:gap-bound}.
\end{proof}

\section{The critical second-variation channel}\label{sec:spectral}

Let $(S,T)\in\mathcal A_{\mathrm p}$ be a non-negative diagonal planar
minimizer, and set $\mathbf Q_p=\mathbf Q^d(S,T)$.  Consider the
off-diagonal perturbation
\begin{equation}\label{eq:Pp}
\mathbf P_p(z)=\begin{pmatrix}0&0&p(z)\\0&0&0\\p(z)&0&0\end{pmatrix}.
\end{equation}
Expanding the definition \eqref{eq:second-variation-definition} using
\eqref{eq:intro-one-dimensional-energy} gives
\begin{align}\label{eq:general-second-variation}
E_{\mathbf Q}(\mathbf P,\mathbf P)=\int_\R\biggl\{&\frac{1}{6}|\mathbf P'|^2
+\frac{L}{6}\sum_{i=1}^3((\mathbf P')_{i3})^2+\frac{1}{6}|\mathbf P|^2
-3\tr(\mathbf P^2\mathbf Q)+(\mathbf Q:\mathbf P)^2+\frac{1}{2}|\mathbf Q|^2|\mathbf P|^2\biggr\}\mathrm{d} z.
\end{align}
Substitution of \eqref{eq:Pp} and $\mathbf Q=\mathbf Q_p$ into
\eqref{eq:general-second-variation} gives
\begin{equation}\label{eq:I31}
E_{\mathbf Q_p}(\mathbf P_p,\mathbf P_p)=\frac{1}{3}\mathfrak q_L[p],
\end{equation}
where $\mathfrak q_L$, $V_L$, and $\mathscr L_L$ are defined in
\eqref{eq:critical-form}--\eqref{eq:critical-operator}.  Since
$L>-\frac{3}{2}$, \eqref{eq:VL} gives $a_L>0$.  Proposition~\ref{prop:existence}
and \eqref{eq:planar-bc} give
\[
V_L(-\infty)=1,\quad V_L(+\infty)=0.
\]
Set
\[
V_{\mathrm{step}}(z):=\boldsymbol{1}_{(-\infty,0)}(z),\quad
\mathscr L_{\mathrm{step}}
:=-a_L\frac{\mathrm{d}^2}{\mathrm{d} z^2}+V_{\mathrm{step}}(z).
\]
Equations \eqref{eq:VL} and \eqref{eq:profile-decay} give
\[
\lim_{z\to-\infty}(V_L-V_{\mathrm{step}})=0,\quad
\lim_{z\to+\infty}(V_L-V_{\mathrm{step}})=0.
\]
Multiplication by $V_L-V_{\mathrm{step}}$ is compact from $H^2(\R)$ to
$L^2(\R)$ by local Rellich compactness and
\begin{equation}\label{eq:compact-perturbation-estimate}
\|(V_L-V_{\mathrm{step}})u\|_{L^2(\R\backslash[-R,R])}
\leq\|V_L-V_{\mathrm{step}}\|_{L^\infty(\R\backslash[-R,R])}
\|u\|_{L^2}.
\end{equation}
The Kato--Rellich theorem \cite[Theorem 6.4]{Teschl2014} gives
\[
\mathcal D(\mathscr L_{\mathrm{step}})=H^2(\R),\quad
\mathscr L_{\mathrm{step}}=\mathscr L_{\mathrm{step}}^*.
\]
For $u\in H^1(\R)$,
\[
\mathfrak q_{\mathrm{step}}[u]
:=a_L\int_\R|u'|^2\mathrm{d} z
+\int_{-\infty}^0|u|^2\mathrm{d} z\geq0.
\]
If
$\chi\in C_0^\infty((1,2))$, $\|\chi\|_{L^2}=1$, and $k\geq0$, then
\[
\psi_n(z):=n^{-\frac{1}{2}}\chi\left(\frac{z-3n}{n}\right)\exp(\mathrm{i}kz)
\]
is supported in $(4n,5n)$ and satisfies
\[
\|\psi_n\|_{L^2}=1,\quad \psi_n\rightharpoonup0,\quad
\|(\mathscr L_{\mathrm{step}}-a_Lk^2)\psi_n\|_{L^2}\to0,
\quad n\to+\infty.
\]
Weyl's criterion \cite[Lemma 6.17]{Teschl2014} gives
\[
\sigma_{\ess}(\mathscr L_{\mathrm{step}})=[0,+\infty).
\]
The perturbation $V_L-V_{\mathrm{step}}$ is relatively compact by
\eqref{eq:compact-perturbation-estimate}, so Weyl's theorem
\cite[Theorem 6.19]{Teschl2014} gives
\begin{equation}\label{eq:ess-spectrum}
\sigma_{\ess}(\mathscr L_L)=[0,+\infty).
\end{equation}
Here $\sigma_{\ess}$ denotes the spectrum after removing isolated
eigenvalues of finite multiplicity.  Consequently, any negative spectral
point of $\mathscr L_L$ is discrete, and \eqref{eq:Lambda} shows that
$\mathfrak q_L$ has a negative direction if and only if $\Lambda(L)<0$.

For comparison with the earlier trial, \eqref{eq:box-bounds} gives $T>0$.
Setting $p=T\eta$ and using the $T$ equation in \eqref{eq:EL} together
with \eqref{eq:VL} gives
\[
V_L=\frac{T''}{T}-9(T-S).
\]
For $\eta\in C_0^\infty(\R)$, integration by parts yields
\[
\int_\R\left(((T\eta)')^2+\frac{T''}{T}(T\eta)^2\right)\mathrm{d} z
=\int_\R T^2(\eta')^2\mathrm{d} z.
\]
Substitution into \eqref{eq:critical-form} gives the exact ground-state
transform
\begin{equation}\label{ground-state}
\mathfrak q_L[T\eta]
=\int_\R\left\{
\frac{L}{2}((T\eta)')^2
+T^2(\eta')^2-9T^2(T-S)\eta^2
\right\}\mathrm{d} z.
\end{equation}
For $\eta=(T-S)_+:=\max\{T-S,0\}$, \eqref{eq:profile-decay} gives
$T\eta\in H^1(\R)$, and
\[
\int_\R T^2((\eta')^2-9(T-S)\eta^2)\mathrm{d} z
=\int_{\{T>S\}}T^2((T'-S')^2-9(T-S)^3)\mathrm{d} z.
\]
Since $L<0$, the first term in \eqref{ground-state} is non-positive.
Thus, \eqref{eq:old-condition} tests only $p=T(T-S)_+$; smooth compact
cutoffs are justified by \eqref{eq:profile-decay}.  Formula
\eqref{eq:Lambda}, not this single substitution, is the variationally
optimal quantity.

\section{A rotational negative mode for all \texorpdfstring{$L<0$}{L<0}}
\label{sec:instability}

\begin{lemma}\label{lem:derivative-order}
Suppose $0<\alpha<1$.  Then
\begin{equation}\label{eq:derivative-order}
3T'(z)-S'(z)>0,\quad z\in\R.
\end{equation}
\end{lemma}

\begin{proof}
Use the coefficients $A,B,C,D$ defined in
\eqref{eq:Jacobi-coefficients}, and put
\[
w=S'-3T'.
\]
From the differentiated system \eqref{eq:Jacobi}, after substituting
$S'=w+3T'$, we obtain
\begin{equation}\label{eq:w-equation}
w''-c_0w=F_0T',
\end{equation}
where
\[
c_0=\frac{A}{\alpha}-3C,
\quad
F_0=\frac{3A+B}{\alpha}-9C-3D.
\]
By \eqref{eq:Jacobi-coefficients} and \eqref{eq:box-bounds},
\[
A>0,\quad C<0,
\]
so $c_0>0$.

It remains to prove $F_0>0$.  Lemma~\ref{lem:gap} yields
$T-S<\frac{1}{6}$, and hence
\begin{align*}
3A+B
&=3(1+6S+6S^2+6T^2+4ST-6T)\\
&=3((1+6S-6T)+6S^2+6T^2+4ST)>0.
\end{align*}
Because $0<\alpha<1$,
\begin{align*}
F_0
&>(3A+B)-(9C+3D)\\
&=12(S(3-2T)+S^2+3T(1-T))>0,
\end{align*}
where \eqref{eq:box-bounds} was used in the last step.
Thus,
\begin{equation}\label{eq:w-coefficient-signs}
c_0>0,\quad F_0>0,\quad z\in\R.
\end{equation}

By \eqref{eq:profile-decay},
\[
\lim_{z\to-\infty}w(z)=0,\quad
\lim_{z\to+\infty}w(z)=0.
\]
Assume that $w(\bar z)>0$ for some $\bar z\in\R$, and set
\[
M:=\sup_{z\in\R}w(z)\geq w(\bar z)>0.
\]
Proposition~\ref{prop:existence} and \eqref{eq:profile-decay} give
$M<+\infty$ and $R>|\bar z|$ such that
\begin{equation}\label{eq:w-tail-bound}
|w(z)|<\frac{M}{2},\quad |z|\geq R.
\end{equation}
Since $w$ is continuous by Proposition~\ref{prop:existence}, it attains its
maximum on $[-R,R]$.  By the definition of $M$ and
\eqref{eq:w-tail-bound},
\[
M=\max_{z\in[-R,R]}w(z).
\]
Equation~\eqref{eq:w-tail-bound} gives $w(-R)<M$ and $w(R)<M$.  Hence there
exists $z_m\in(-R,R)$ such that
\[
w(z_m)=M,
\quad w'(z_m)=0,
\quad w''(z_m)\leq0.
\]
Using \eqref{eq:w-coefficient-signs},
\[
w''(z_m)-c_0(z_m)w(z_m)
\leq-c_0(z_m)M<0.
\]
On the other hand, Proposition~\ref{prop:monotone} and
\eqref{eq:w-coefficient-signs} give
\[
F_0(z_m)T'(z_m)>0,
\]
contradicting \eqref{eq:w-equation}.  Hence $w\leq0$ on $\R$.

Assume that $w(z_0)=0$ for some $z_0\in\R$.  Since $w\leq0$, the point
$z_0$ is a global maximum, and
\[
w'(z_0)=0,\quad w''(z_0)\leq0.
\]
Equations \eqref{eq:w-equation} and \eqref{eq:w-coefficient-signs}, together
with Proposition~\ref{prop:monotone}, give
\[
w''(z_0)
=c_0(z_0)w(z_0)+F_0(z_0)T'(z_0)
=F_0(z_0)T'(z_0)>0,
\]
a contradiction.  Therefore $w<0$ on $\R$, which is
\eqref{eq:derivative-order}.
\end{proof}

The scalar test function $p$ in the off-diagonal perturbation
\eqref{eq:Pp} is generated by rotation.  The eigenvalues of $\mathbf Q_p$
along $\mathbf e_1$ and $\mathbf e_3$ are
\[
\lambda_1=\frac{S+3T}{3},
\quad
\lambda_3=-\frac{2S}{3},
\]
and therefore
\[
\lambda_1-\lambda_3=S+T.
\]
Let
\[
\mathbf R_\theta:=
\begin{pmatrix}
\cos\theta&0&-\sin\theta\\
0&1&0\\
\sin\theta&0&\cos\theta
\end{pmatrix}
\]
be the rotation through angle $\theta$ in the
$(\mathbf e_1,\mathbf e_3)$ plane.  Differentiating the rotated tensor
$\mathbf R_\theta \mathbf Q_p\mathbf R_\theta^{\mathrm{T}}$ at $\theta=0$ gives
\[
\left.\frac{\mathrm{d}}{\mathrm{d}\theta}
(\mathbf R_\theta \mathbf Q_p\mathbf R_\theta^{\mathrm{T}})\right|_{\theta=0}
=(S+T)(\mathbf e_1\otimes \mathbf e_3+\mathbf e_3\otimes \mathbf e_1).
\]
Comparing this expression with \eqref{eq:Pp} shows that the scalar mode is
$p=S+T$.  Thus this perturbation is precisely
the infinitesimal rotation of the planar director toward the interface
normal.

\begin{lemma}\label{lem:rotation}
For any solution of \eqref{eq:EL},
\begin{equation}\label{eq:rotation-operator}
\mathscr L_L(S+T)=\frac{L}{6}(S''-3T'').
\end{equation}
For the non-negative diagonal planar minimizer in
Proposition~\ref{prop:existence} with $-\frac{3}{2}<L<0$,
\begin{equation}\label{eq:rotation-energy}
\begin{aligned}
\mathfrak q_L[S+T]
:&=\int_\R\{a_L(S'+T')^2+V_L(S+T)^2\}\mathrm{d} z\\
&=\frac{L}{6}\int_\R(S'+T')(3T'-S')\mathrm{d} z<0.
\end{aligned}
\end{equation}
\end{lemma}

\begin{proof}
Adding the two equations in \eqref{eq:EL} and using \eqref{eq:VL} gives
\[
\begin{aligned}
\alpha S''+T''=S+T+3S^2-6ST-9T^2+2S^3+2S^2T+6ST^2+6T^3=V_L(S+T).
\end{aligned}
\]
Since $a_L=1+\frac{L}{2}$ and $\alpha=1+\frac{2L}{3}$,
\begin{align*}
\mathscr L_L(S+T)
&=-a_L(S''+T'')+\alpha S''+T''\\
&=(\alpha-a_L)S''+(1-a_L)T''\\
&=\frac{L}{6}(S''-3T''),
\end{align*}
which proves \eqref{eq:rotation-operator}.  Equation~\eqref{eq:profile-decay}
and the definition \eqref{eq:VL} give
\[
\lim_{z\to-\infty}(S'+T')(z)=0,\quad
\lim_{z\to+\infty}(S'+T')(z)=0,
\]
and
\[
(S'+T')^2+|V_L|(S+T)^2+(S+T)|S''-3T''|\in L^1(\R).
\]
Thus \eqref{eq:rotation-operator} implies
$(S+T)\mathscr L_L(S+T)\in L^1(\R)$ even though $S+T\to1$ at
$z\to+\infty$ by \eqref{eq:planar-bc}.  Moreover,
\eqref{eq:profile-decay} and \eqref{eq:planar-bc} give
\[
\lim_{z\to-\infty}(S+T)(S'+T')=0,\quad
\lim_{z\to+\infty}(S+T)(S'+T')=0.
\]
Integration by parts in \eqref{eq:critical-form} therefore gives
\begin{align*}
\mathfrak q_L[S+T]
&=\int_\R(S+T)\mathscr L_L(S+T)\mathrm{d} z\\
&=-\frac{L}{6}\int_\R(S'+T')(S'-3T')\mathrm{d} z\\
&=\frac{L}{6}\int_\R(S'+T')(3T'-S')\mathrm{d} z.
\end{align*}
By Proposition~\ref{prop:monotone} and \eqref{eq:derivative-order},
\[
(S'+T')(3T'-S')>0,\quad z\in\R.
\]
Since $L<0$, this proves
\eqref{eq:rotation-energy}.
\end{proof}

\begin{proof}[Proof of Theorem~\ref{thm:main}]
By Lemma~\ref{lem:rotation} and \eqref{eq:rotation-energy}, the mode
$p_0=S+T$ satisfies $\mathfrak q_L[p_0]<0$.  By
\eqref{eq:planar-bc},
\[
\lim_{z\to+\infty}p_0(z)=1,
\]
so $p_0\notin L^2(\R)$.  We therefore localize it.  Choose smooth cutoffs
$0\leq\chi_{R,M}\leq1$ such that $\chi_{R,M}=1$ on $[-R,R]$,
$\chi_{R,M}=0$ on $(-\infty,-2R]\cup[R+M,+\infty)$, and
\[
|\chi_{R,M}'|\leq \frac{C}{R}\quad\text{on }[-2R,-R],
\quad
|\chi_{R,M}'|\leq \frac{C}{M}\quad\text{on }[R,R+M].
\]
Expanding \eqref{eq:critical-form} and using
\eqref{eq:critical-operator}, integration by parts gives the exact
localization identity
\[
\mathfrak q_L[p_0\chi_{R,M}]
=\int_\R \chi_{R,M}^2p_0\mathscr L_Lp_0\mathrm{d} z
+a_L\int_\R(\chi_{R,M}')^2p_0^2\mathrm{d} z.
\]
Set $g:=p_0\mathscr L_Lp_0$.  Proposition~\ref{prop:existence} and
\eqref{eq:rotation-operator} give $g\in L^1(\R)$,
$|p_0(z)|\leq C\exp(cz)$ for $z\leq0$, and $\|p_0\|_{L^\infty}\leq C$.
Consequently,
\[
\left|\int_\R(\chi_{R,M}^2-1)g\mathrm{d} z\right|
\leq\int_{\R\backslash[-R,R]}|g|\mathrm{d} z,
\]
and
\[
a_L\int_\R(\chi_{R,M}')^2p_0^2\mathrm{d} z
\leq C a_L\left(\frac{\exp(-2cR)}{R}+\frac{1}{M}\right).
\]
Both right-hand sides tend to zero as $R,M\to+\infty$.  Hence
\[
\lim_{R,M\to+\infty}
\mathfrak q_L[p_0\chi_{R,M}]=\mathfrak q_L[p_0]<0.
\]
Hence some smooth compactly supported $p$ satisfies
$\mathfrak q_L[p]<0$.  With $\mathbf P=\mathbf P_p$ from \eqref{eq:Pp}, equation
\eqref{eq:I31} gives $E_{\mathbf Q_p}(\mathbf P,\mathbf P)<0$.
\end{proof}

An eigenvalue $\lambda$ of $\mathscr L_L$ is called simple if
\begin{equation}\label{eq:simple-eigenvalue-definition}
\dim\{p\in H^2(\R):\mathscr L_Lp=\lambda p\}=1.
\end{equation}

\begin{corollary}
For any $L\in(-\frac{3}{2},0)$, the number $\Lambda(L)<0$ is a simple
eigenvalue of $\mathscr L_L$ with a strictly positive eigenfunction.
\end{corollary}

\begin{proof}
Theorem~\ref{thm:main}, \eqref{eq:I31}, and the variational characterization
\eqref{eq:Lambda} give $\Lambda(L)<0$.  Equation~\eqref{eq:ess-spectrum}
and the max--min principle \cite[Theorem 4.12]{Teschl2014} give
\[
\Lambda(L)<\inf\sigma_{\ess}(\mathscr L_L)=0
\quad\Longrightarrow\quad
\Lambda(L)\in\sigma_{\mathrm d}(\mathscr L_L).
\]
Here $\sigma_{\mathrm d}$ denotes the isolated eigenvalues of finite
multiplicity.  Since $V_L\in L^\infty(\R)$ by \eqref{eq:VL} and
Proposition~\ref{prop:existence}, the ground-state theorem
\cite[Theorem 10.13]{Teschl2014} applies to
\[
a_L^{-1}\mathscr L_L
=-\frac{\mathrm{d}^2}{\mathrm{d} z^2}+a_L^{-1}V_L.
\]
Consequently,
\[
\dim\{p\in H^2(\R):\mathscr L_Lp=\Lambda(L)p\}=1,
\]
so $\Lambda(L)$ is simple by \eqref{eq:simple-eigenvalue-definition}.  The
same theorem gives an eigenfunction $p\in H^2(\R)$ satisfying
\[
p(z)>0,\quad z\in\R,
\quad -a_Lp''+V_Lp=\Lambda(L)p\quad\text{in }L^2(\R).
\]
This completes the proof.
\end{proof}

\section{The threshold \texorpdfstring{$L=0$}{L=0}}\label{sec:threshold}

At $L=0$ the planar profile is explicit and the critical operator admits an
exact factorization.  A zero-energy threshold resonance is a non-zero
bounded solution of $\mathscr L_0p=0$ that is not in $L^2(\R)$, when zero is
the bottom of $\sigma_{\ess}(\mathscr L_0)$.

\begin{proposition}[Zero-energy resonance]\label{prop:zero-resonance}
Let $L=0$ and let $s$ be given by \eqref{eq:correct-planar-L0}.  Then
$S=T=s$ and
\begin{equation}\label{eq:L0-factorization}
\mathfrak q_0[p]=\int_\R s^2\left[\left(\frac{p}{s}\right)'\right]^{2}\mathrm{d} z\geq0
,\quad p\in C_0^\infty(\R).
\end{equation}
Consequently $\Lambda(0)=0$.  The positive solution $p=s$ of
$\mathscr L_0p=0$ satisfies
\[
s\in L^\infty(\R)\backslash L^2(\R),\quad
\lim_{z\to-\infty}s(z)=0,\quad
\lim_{z\to+\infty}s(z)=\frac{1}{2}.
\]
Since $0=\inf\sigma_{\ess}(\mathscr L_0)$ by
\eqref{eq:ess-spectrum}, $s$ is a zero-energy threshold resonance and not
an eigenfunction.
\end{proposition}

\begin{proof}
For $S=T=s$, equations \eqref{eq:correct-planar-L0} and \eqref{eq:VL} give
\[
V_0=1-6s+8s^2=\frac{s''}{s}.
\]
Therefore, \eqref{eq:critical-operator} yields
\[
\mathscr L_0s=-s''+V_0s=0.
\]
Since $a_0=1$, substituting $p=s\eta$ and integrating the cross term by
parts in \eqref{eq:critical-form} gives
\[
\int_\R\left((p')^2+\frac{s''}{s}p^2\right)\mathrm{d} z
=\int_\R s^2(\eta')^2\mathrm{d} z,
\]
which is \eqref{eq:L0-factorization}.  Density of
$C_0^\infty(\R)$ in $H^1(\R)$ gives $\Lambda(0)\geq0$ by
\eqref{eq:Lambda}.  Conversely, choose $\chi_R\in C_0^\infty(\R)$ such that
\[
0\leq\chi_R\leq1,\quad \chi_R=1\ \text{on }[-R,R],\quad
\chi_R=0\ \text{outside }[-2R,2R],\quad
|\chi_R'|\leq\frac{C}{R}.
\]
Since $s(z)\to\frac{1}{2}$ as $z\to+\infty$, \eqref{eq:L0-factorization}
gives
\[
\mathfrak q_0[s\chi_R]=\int_\R s^2(\chi_R')^2\mathrm{d} z
\leq\frac{C}{R},\quad
\|s\chi_R\|_{L^2}^2\geq cR.
\]
Hence
\[
0\leq\Lambda(0)
\leq\frac{\mathfrak q_0[s\chi_R]}{\|s\chi_R\|_{L^2}^2}
\leq\frac{C}{R^2}\to0,\quad R\to+\infty.
\]
This proves $\Lambda(0)=0$.  Finally,
\eqref{eq:correct-planar-L0} gives $s\in L^\infty(\R)\backslash L^2(\R)$,
and \eqref{eq:ess-spectrum} gives
$0=\inf\sigma_{\ess}(\mathscr L_0)$.  Thus $s$ has all the properties
stated in Proposition~\ref{prop:zero-resonance}.
\end{proof}

Thus the negative-$L$ instability reaches the endpoint $L=0$ from below.
There is no positive spectral gap at the endpoint, and perturbation theory near
$L=0$ is a resonance problem rather than ordinary isolated-eigenvalue
perturbation theory.

\section*{Acknowledgments} 

This work was partially supported by the National Key R$\&$D Program of China under Grant 2023YFA1008801.

\end{document}